\documentclass[12pt]{article}
\usepackage[utf8]{inputenc}
\usepackage[a4paper, total={6.5in, 9.5in}]{geometry}

\usepackage[reqno]{amsmath}
\counterwithin{equation}{section}
\usepackage{exscale} 
\usepackage[colorlinks=true, pdfstartview=FitV, linkcolor=blue, citecolor=red, urlcolor=black,pagebackref=false]{hyperref}
\usepackage{fontenc}
\usepackage{makeidx}
\usepackage{amsfonts}
\usepackage{amstext}
\usepackage{amsrefs}
\usepackage{amsthm}
\usepackage{amssymb}
\usepackage{bbm}
\usepackage[utf8]{inputenc}
\usepackage{xcolor}
\usepackage{tikz} 
\usepackage{appendix}
\usepackage{booktabs}
\usepackage{float}
\usepackage{esint}
\usepackage{authblk}
\usepackage{fancyhdr}

\newcommand{\R}{\mathbb{R}}
\newcommand{\N}{\mathbb{N}}
\newcommand{\D}{\Delta}

\theoremstyle{plain}
\newtheorem{theo}{Theorem}[section]
\newtheorem{prop}[theo]{Proposition}
\newtheorem{lemma}[theo]{Lemma}

\theoremstyle{remark}
\newtheorem{rem}{Remark}

\newtheorem*{notation}{Notation}

\theoremstyle{definition}

\allowdisplaybreaks

\title{\textsc{\Large{On the definition of zero resonances for the Schr\"odinger operator with optimal scaling
			potentials}}}
\date{}
\author{\textsc{\normalsize Viviana Grasselli}}

\fancypagestyle{footer}{\fancyhf{}
	\lfoot{\small \textbf{2020 Mathematics Subject Classification:} 35J10, 35B34, 35P25}}

\begin{document}
	
	\maketitle
	\begin{abstract}
		
		We consider the Schr\"{o}dinger operator $-\Delta + V$ on $\mathbb{R}^n$ with potential in the Lorentz space $L^{n/2,1}$ and we find necessary and sufficient conditions for zero to be a resonance or an eigenvalue. 
		We consider functions with gradient in $L^2$ and that verify the equation $(-\D + V)\psi =0$, namely $\ker_{\dot{H}^1}(-\Delta +V)$. We prove that a function in this set is either in a weak Lebesgue space, $L^{\frac{n}{n-2},\infty}$, or in $L^2$, in the latter case we have a zero eigenfunction. The set of eigenfunctions is the hyperplane of functions that are orthogonal to $V$, furthermore we show that under some classic orthogonality conditions a zero eigenfunction belongs to $L^{1, \infty}$ or $L^1$. We study dimensions $n\geq 3$ and in dimension three we generalize a result proved by Beceanu. 
		
		%    We consider the Schr\"{o}dinger operator $-\Delta + V$ on $\mathbb{R}^3$ with potential in the Lorentz space $L^{3/2,1}$. We find necessary and sufficient conditions for a function in the kernel of $-\Delta +V$ to be a resonance or an eigenstate at 0. In particular we obtain\\ $\ker_{\dot{H}^1}(-\Delta +V)= \ker_{\dot{H}^1}(-\Delta +V)\cap L^{3, \infty} \sqcup \ker_{\dot{H}^1}(-\Delta +V)\cap L^2$ with $\ker_{\dot{H}^1}(-\Delta +V)\cap L^2 
		%    $ the hyperplane of functions that are orthogonal to $V$.  This property generalizes a result in \cite{beceanu_2016}.
	\end{abstract}

	\thispagestyle{footer}
	\section{Introduction}

	We consider the Schr\"{o}dinger operator $-\Delta + V$ on $\mathbb{R}^n$ with $n\geq 3$. The evolution of solutions of the time-dependent Schr\"{o}dinger equation is influenced by the the spectrum of the operator $-\Delta +V$, in particular by the bottom of it and the nature of the zero frequency: whether it is in the spectrum and if it is an eigenvalue or a resonance. Roughly speaking, zero is a resonance if there is a solution to the equation $(-\Delta + V)\psi=0$ which does not decay fast enough to be in $L^2$ and is usually assumed to belong in some kind of weighted space. This solution is a resonant state. 
	
	Local energy decay for solutions of dispersive equations, for example, is tightly linked to the spectral properties of the Schr\"odinger operator. First results in this direction were proved in three dimensional weighted $L^2$ spaces in \cite{rauch}, \cite{jen_kato}, in higher dimension in \cite{jensen_d4}, \cite{jensen_dim>5} and for more general elliptic operators in \cite{murata}. For an extensive overview on the subject the reader may refer to the survey article \cite{schlag_disp_survey}. A way to approach these estimates is through expansions of the resolvent around zero energy from which one can obtain an expansion of the evolution operator via Fourier transform. This is the strategy followed in \cite{jen_kato}, where the authors compute resolvent expansions in dimension three distinguishing the cases whether there are eigenvalues or resonance obstructions at zero. A unified approach  for resolvent expansions applicable to all dimensions was found later in \cite{jensen_nenciu}.

	In this brief note we recover properties of zero resonances and eigenfunctions described in the seminal paper \cite{jen_kato} under weaker assumptions on $V$ and with optimal behavior of such states as $|x|$ tends to infinity. On top of this, we obtain results of the same flavor in dimension four, covering all cases where zero resonances exist and showing that for $n\geq 5$ our solutions always belong to $L^2$. We do not exhibit any exotic or new behavior but rather give what we consider to be a fairly simple proof of classical results for a more general class of potentials (essentially the optimal one). Our result in dimension three also has the advantage of generalizing a characterization found by Beceanu in \cite{beceanu_2016} as we will see more in detail later on.
	
	% The aim of this brief note is to show how for a very natural choice of potentials we can derive properties on the 0 resonance which somehow simplify its definition. In particular we will give a necessary and sufficient condition for the resonance to be in $L^2$ and therefore being an eigenfunction. If this condition is not met we will obtain that the resonance is in a weak Lebesgue space. Our result also has the advantage of generalising a characterisation found by Beceanu in \cite{beceanu_2016} as we will see more in detail later on.  
	% The fact that we can consider a 0 resonance as a function in a weak Lebesgue spaces is a very useful feature and it comes at a very small cost. Indeed, in much of the literature resonances are usually found to belong to weighted Sobolev or Lebesgue spaces where there is no translation invariance.
	
	In much of the literature, resonances are usually found to belong to weighted Sobolev or Lebesgue spaces where there is no scaling invariance. For example, in \cite{jen_kato} the authors define a resonance as a solution to $(-\Delta + V)\psi=0$ where the operator $-\Delta + V$ is meant to be extended to the weighted Sobolev space $H_{1, -s}$ with weight $\langle x \rangle^{-s}$ and a suitable $s>1/2$.  A similar definition can also be found in the leading paper on dispersive $L^1$ to $L^\infty$ estimates by Journé, Soffer and Sogge (\cite{journe_soffer_sogge}). Therein a zero resonant state is taken in a weighted $L^2$ space, which is the same definition used in \cite{yajima} in which the author studies $L^p$ to $L^q$ dispersive estimates. Similarly, in \cite{rodn_schlag},\cite{goldberg}, \cite{goldberg2} and \cite{goldberg_schlag_dim1_3} the authors define resonances as functions that belong to the intersection of weighted $L^2$ spaces given by $\cap_{s>1/2} L^{2,-s} $ with weight $\langle x \rangle^{-s}$. Even in more recent papers (\cite{aafarani} ,\cite{soffer_wu}) with much stricter assumptions on the potentials than the one considered herein, the framework used to define resonances is the one of weighted $L^2$ spaces.
	Along the same line, one can consider the definition in \cite{tataru}. For a generic time dependent second order elliptic operator a zero resonant state (Definition 1.8) is taken in a rather complicated local smoothing space that for the case of dimension three and higher is contained in the weighted $L^2$ space with weight $\langle x \rangle ^{-1}$. As an example for the non euclidean case, in \cite{wang} on a conical manifold we have a perturbation of a model operator both by a potential and a metric and resonances are defined as functions in weighted Sobolev spaces. In such work the author, as in \cite{jen_kato}, is mainly interested in the influence of zero resonant states on the singularities of the resolvent.

	The fact that we can consider a zero resonant state as a function in a weak Lebesgue space is a very useful feature and it comes at a fairly small cost. Indeed, in various applications (like \cite{jen_kato}, \cite{goldberg_schlag_dim1_3}, \cite{erdogan_schlag}, \cite{aafarani} or \cite{soffer_wu}) the potentials must have a prescribed decay at infinity, that is $|V(x)|\lesssim\langle x \rangle^{-\beta}$ for some $\beta>2$. Our assumption will be much less strict than requiring a specific pointwise decay and the results presented here only consider scaling invariant classes for the potentials. This can be a useful feature when working with nonlinear problems and we aim to further investigate properties on the bottom of the spectrum for Schr\"{o}dinger operators with potentials that satisfy only scaling invariant assumptions. 
	
	% , since in general $\frac{1}{|x|^l} \in L^p(B(0,1)^c)$ if and only if $p>\frac{3}{l}$.
	
	% This last feature is very useful to clarify the definition of 0 energy resonances. Indeed, in the literature resonances are usually defined as belonging to weighted Sobolev or Lebesgue spaces, ha
	%(which are Banach spaces in this case?..). 
	% For example, in \cite{jen_kato} the authors define a resonance as a solution to $(-\Delta + V)u=0$ where the operator $-\Delta + V$ is meant to be extended to the weighted Sobolev space $H_{1, -s}$ with weight $\langle x \rangle^{-s}$ and a suitable $s>1/2$. 
	
	% The same definition is also used by the authors in \cite{schlag_rodnianski},\cite{goldberg} and \cite{goldberg2}, where they define resonances as functions that belong to the intersection of weighted $L^2$ spaces given by $\cap_{\alpha>1/2} L^{2,-\alpha} $ with weight $\langle x \rangle^{-\alpha}$.  And in general resonances are taken as poles of the meromorphic extension of the resolvent defined on weighted spaces. 
	
	% (EXPAND) 

	% The potentials that are considered in the previously mentioned papers by \cite{jen_kato} and \cite{schlag_rodnianski} are of the form $V(x)= O(\langle x \rangle^{-2-\varepsilon})$ (however we remark that for several results in \cite{jen_kato} is required $V(x)= O(\langle x \rangle^{-\beta})$ with $\beta\geq 3$ check!). A potential of the form $\langle x \rangle^{-2-\varepsilon}$ is in $ L^{3/2}(\mathbb{R}^3)$ at best, since in general $\frac{1}{|x|^l} \in L^p(B(0,1)^c)$ if and only if $p>\frac{3}{l}$.
	
	More precisely, let $L^{p,q}(\mathbb{R}^n)$ the Lorentz space defined on $\mathbb{R}^n$.  It is defined as the space of functions such that the quasinorm denoted by $\|\cdot\|_{p,q}$ is finite. We define the quasinorm via the distribution function $ d_f(t)= |\{|f(x)|>t\}|$ as
	\begin{align*}
		\|f\|_{p,q}:= p^{1/q}\left(\int_0^\infty t^{q-1} (d_f(t))^{q/p} dt \right)^{1/q} 
	\end{align*}
	for $q<\infty$ or 
	\begin{equation*}
		\| f\|_{p, \infty} := \sup_{t\geq 0} t d_f(t)^{1/p}
	\end{equation*}
	otherwise. As for the Lebesgue spaces, we identify two functions that are equal almost everywhere. (For some properties that we will use throughout the reader can refer to Appendix \ref{section: appendix lorentz spaces}).

	As mentioned before, for a prescribed decay the most general assumption is $V(x)$ decaying like $|x|^{-2-\varepsilon}$ at infinity (notice that $\langle x \rangle^{-2-\varepsilon}\in L^{n/2}(\mathbb{R}^n)$). From now on, we will take $V$ in a slightly smaller space than $L^{n/2}(\mathbb{R}^n)$, that is we will assume 
	\begin{equation}
		\label{def: V in L^3/2,1}
		V\in L^{n/2, 1}(\mathbb{R}^n)\subset L^{n/2}(\mathbb{R}^n).
	\end{equation}
	Assuming  \eqref{def: V in L^3/2,1} is a little more restrictive than taking the potential in a Lebesgue space or in a Kato class, like in \cite{bui_duong_hong} and \cite{beceanu_goldberg}. However, it is still less strict than what is assumed in \cite{goldberg}, \cite{goldberg2}, or \cite{goldberg_res_weak}, where $L^{3/2-\varepsilon}\cap L^{3/2+\varepsilon}\subset L^{3/2,1} $ from Proposition \ref{prop: f in intersection of L^p infty}, and much more general than imposing a power-like decay, which we have seen is the custom in several cases. Moreover, our assumption is satisfied from the case $V(x)= O(\langle x \rangle^{-2-\varepsilon})$, which is a common assumption in the low frequencies asymptotic. Nonetheless, it does not seem to have been clearly observed that a larger class, such as $L^{n/2, 1}(\mathbb{R}^n)$, gives the right framework to describe zero resonances and eigenfunctions more simply and optimally than in the usual approach \cite{jen_kato}. Indeed, the results presented in Theorems \ref{th: res in weak sp} and \ref{th: properties of 0 resonance} may already be known to some experts, but we couldn't find any reference in the literature.

	We will be interested in $\dot{H}^1(\mathbb{R}^n)$ solutions of the equation
	\begin{equation}
		\label{eq: equation of the resonance}
		(-\Delta +V)\psi=0
	\end{equation}
	with $V\in L^{n/2,1}(\mathbb{R}^n)$ and $\dot{H}^1(\mathbb{R}^n)$ the homogeneous Sobolev space (see \eqref{def: hom sobolev} for a definition). We will state our result for a function in  $\dot{H}^1(\mathbb{R}^n)$ which gives a simple and natural class of non necessarily $L^2$ functions where to look for solutions of \eqref{eq: equation of the resonance}. Moreover, $-\Delta$ is an isometry from $\dot{H}^1(\mathbb{R}^n)$ to $\dot{H}^{-1}(\mathbb{R}^n)$, a feature that we shall extensively use in Section \ref{section: green function for small W}.

	We also comment that by inverting $-\Delta$, equation \eqref{eq: equation of the resonance} is equivalent to solving
	\begin{equation*}
		\psi + (-\Delta)^{-1}V\psi=0.
	\end{equation*}
	In his paper \cite{beceanu_2016} Beceanu solves this equation in $L^\infty(\mathbb{R}^3)$. As we just said, here we will rather solve it in $\dot{H}^1(\mathbb{R}^n)$, or actually in $L^{\frac{2n}{n-2}, \infty}(\R^n)$, in particular without seeking a priori bounded solutions.

	The main results we will prove are the following. 
	
	\begin{theo}
		\label{th: res in weak sp}
		Let $n\geq 3$, $V \in L^{n/2,1}(\R^n)$ and $\psi \in \dot{H}^1(\mathbb{R}^n)$ a solution of the equation $(-\D + V)\psi =0$.
		\begin{itemize}
			\item [$i)$] If $n \geq 5$ then $\psi \in L^2$ and hence all solutions are eigenfunctions. 
			\item [$ii)$] If $n= 3,4$ then
			\begin{equation*}
				\lim_{|x| \to \infty} |x|^{n-2}\psi(x)= -c_n\int V(y) \psi(y)dy <\infty,
			\end{equation*}
			with $c_n= n(n-2)|B(0,1)|$. 
			Hence for $n=3$
			\begin{equation*}
				\psi \in L^{3, \infty}(\R^3)
			\end{equation*}
			and for $n=4$ 
			\begin{equation*}
				\psi \in L^{2, \infty}(\R^4).
			\end{equation*}
		\end{itemize}
	\end{theo}
	\begin{rem}
		The conclusion in item $i)$ can also be seen as a consequence of the decay given by item $ii)$. However, we can give a direct proof of the fact $\psi \in L^2$ for $n\geq 5$, without the need of computing the explicit limit of $|x|^{n-2}\psi(x)$. Moreover, item $i)$ is true under the weaker assumption $V\in L^{n/2}$. 
	\end{rem}
	
	\begin{theo}
		\label{th: properties of 0 resonance}
		Let $n\geq 3$, $V \in L^{n/2,1}(\R^n)$ and $\psi \in \dot{H}^1(\mathbb{R}^n)$ a solution of the equation $(-\D + V)\psi =0$. The following properties hold: 
		\begin{itemize}
			\item [$i)$] If $\int V \psi =0$ then $\psi  =O(\frac{1}{|x|^{n-1}})$ near infinity and in particular $\psi $ is in $L^{2}(\R^n)$ and is a zero eigenfunction. Moreover, for $n=3,4$, $\psi \in L^{2}(\R^n)$ if and only if $\int V \psi =0$.
			\item [$ii)$] If $\int V \psi = \int y_k V\psi =0$ for all $k=1, \ldots, n$ then $\psi =O(\frac{1}{|x|^{n}})$ near infinity. In particular, $\psi $ is a zero eigenfunction and $\psi \in L^{1, \infty}(\R^n)$.
			\item [$iii)$] If $\int V \psi = \int y_k V\psi = \int y_k y_l V\psi =0$ for all $k,l=1, \ldots, n$ then $\psi = O( \frac{1}{|x|^{n+1}})$ near infinity. In particular, $\psi $ is a zero eigenfunction and $\psi \in L^{1}(\R^n)$.
			
		\end{itemize}
	\end{theo}

	Let us give a few comments about these results: 
	\begin{itemize}
		\item as we mentioned earlier, $\dot{H}^1(\mathbb{R}^n)$ is a pretty natural class of non $L^2$ functions to consider for this problem and in higher dimension $n\geq 5$ we see from Theorem \ref{th: res in weak sp} that there are no solutions of \eqref{eq: equation of the resonance} which do not belong to $L^2$. This is due to the fact that Sobolev embeddings imply that the operator $I + (-\D)^{-1}V$ maps $L^2$ into itself.  Therefore, with respect to the existence of resonances $\dot{H}^1$ represents an optimal class of functions where to solve \eqref{eq: equation of the resonance}. However, from the point of view of regularity it is actually enough to require $\psi \in L^{\frac{2n}{n-2}, \infty}(\R^n)$ and indeed we will perform all proofs for $\psi $ in this larger class.
		\item In dimension three the simple assumption $V\in L^{3/2,1}(\R^3)$ generalizes the result stated by Beceanu in \cite{beceanu_2016} (Lemma 2.3) where the necessary and sufficient condition for $\psi$ to be an eigenfunction is recovered only for potentials in $L^{3/2,1}(\R^3)\cap L^1(\R^3)$. Regarding the behavior at infinity, we give a more precise statement than just $\psi \in \langle x \rangle^{-1} L^\infty(\R^3)$ since in Theorem \ref{th: res in weak sp} we give an explicit expression for the limit of $|x|\psi$ for $|x|\to +\infty$. 
		
		\item  The conditions of orthogonality between $V\psi$ and various other polynomial functions are not new ones. Indeed, they can also be found in \cite{jen_kato}, \cite{jensen_d4} or \cite{beceanu_2016}. In particular in \cite{jen_kato},\cite{jensen_d4} it was already observed that the condition $\int  V\psi=0$ is the right one to discriminate between zero resonance or eigenvalue. However, in these works the authors consider conditions of pointwise decay on the potential that are more strict than our assumption. We also remark that the orthogonality conditions in item $i)$ and $ii)$ are actually necessary and sufficient for the decay of $\psi$, see Remark \ref{rem: decay psi implies orthogon}. 
		
	\end{itemize}

	\begin{notation}
		We will drop from the notations the dependence on the underlying space $\R^n$ unless the situation requires it to make it explicit. 
		
		We define the function $a$
		\begin{equation*}
			a(x) = \frac{1}{|x|^{n-2}}
		\end{equation*}
		which belongs to $L^{\frac{n}{n-2}, \infty}$. 
	\end{notation}

	The main steps will be the following:
	\begin{enumerate}
		\item
		\label{item: V=W+K}
		We will use the density of simple functions in $L^{n/2,1}$ to decompose the potential $V$ in
		\begin{equation*}
			V=W+K
		\end{equation*}
		where $W\in L^{n/2,1}$ is such that $\|W\|_{n/2,1}$ is arbitrarily small and $K$ is a simple function, hence compactly supported and in any Lorentz space. 
		\item Thanks to the smallness of $W$ in $L^{n/2,1}$ we will be able to define and estimate $G$, the Green function of $(-\Delta +W)$. 
		\item We will then solve in $\dot{H}^1$ (or more in general $L^{\frac{2n}{n-2}, \infty}$ is sufficient) the equation
		\begin{equation*}
			(-\Delta + W) \psi = -K \psi
		\end{equation*}
		where $(-\Delta + W): \dot{H}^1 \rightarrow \dot{H}^{-1}$ is invertible and we will use the Green function computed in the previous step to write 
		\begin{equation*}
			\psi(x)= - \int G(x,y) K(y) \psi(y) dy.
		\end{equation*}
		In the rest of the paper $\psi$ will be a $L^{\frac{2n}{n-2}, \infty}$ solution of $(-\D + V)\psi =0$. 
		
		\item We will show that $|x|^{n-2}\psi$ has finite limit at infinity and that such limit is given by $- c_n\int  V\psi $. Therefore, the value of such integral determines whether $\psi$ is in $L^2$ or not. This is done in Section \ref{section: properties of zero resonance}. 
		
		\item With additional orthogonality conditions on $ V\psi$ we can repeat the argument of the previous step to prove that $\psi$ has faster, but limited, decay until we can reach the $L^1$ space. This is done in Section \ref{section: properties of zero resonance}. 
	\end{enumerate}

	\section{Green function for a small potential}
	\label{section: green function for small W}
	
	Let $\dot{H}^s$ for $s \in \R$ the homogeneous Sobolev space of order $s$, defined as 
	\begin{equation}
		\label{def: hom sobolev}
		\Dot{H}^s:= \{ u \in \mathcal{S}'\ :\ \hat{u} \in L^1_{loc}, \ \|u\|_{\Dot{H}^s}:=\||\cdot|^s\hat{u}\|_{L^2}<\infty\}
	\end{equation}
	where $\mathcal{S}'$ is the space of tempered distributions on $\R^n$. The following Sobolev embedding will be useful
	\begin{equation}
		\label{eq: dot H^1 inclusion}
		\Dot{H}^1\hookrightarrow L^{\frac{2n}{n-2}},
	\end{equation}
	see Theorem 1.38 in \cite{bahouri} for a proof. 
	\begin{rem}
		We remark that the more general embedding 
		\begin{equation}
			\label{eq: dot H^1 in lorentz}
			\Dot{H}^1\hookrightarrow L^{\frac{2n}{n-2}, 2},
		\end{equation}
		is also true. It can be obtained with the same reasoning as in the proof of Theorem 1.38 in \cite{bahouri} and using Young's inequality for Lorentz spaces in dimension $n\geq 3$ (\cite{blozinski} Theorem 2.12):
		\begin{equation*}
			\||\cdot|^{-\alpha} * f\|_{p,q} \leq C \|f\|_{r,q} \quad 1 + \frac{1}{p} = \frac{1}{r} + \frac{\alpha}{n },
		\end{equation*}
		where $|\cdot|^{-\alpha}  \in L^{n /\alpha, \infty}$.

	\end{rem}

	We consider the operator 
	\begin{equation*}
		-\D + V : \dot{H}^1 \to \dot{H}^{-1},
	\end{equation*}
	its	selfadjoint realization on $\dot{H}^1$ is obtained via the quadratic form 
	\begin{equation*}
		q(u,u )= \langle \nabla u, \nabla u\rangle + \langle V u , u \rangle =  \int |\nabla u |^2 + V|u|^2 dx, 
	\end{equation*}
	which is well defined on $\dot{H}^1$. Indeed, given $V \in L^{n/2, 1} \subset L^{n/2}  $ and $ u \in 
	\dot{H}^1\hookrightarrow L^{\frac{2n}{n-2}}$ we have 	
	\begin{equation}
		\label{comp: W maps to H^-1}
		Vu \in L^{n/2} \cdot L^ {\frac{2n}{n-2}}\subset L^{\frac{2n}{n+2}}\subset \Dot{H}^{-1},
	\end{equation}
	thanks to the dual of inclusion \eqref{eq: dot H^1 inclusion}. We have just found that $V$ maps $ \dot{H}^1$ to $ \dot{H}^{-1}$, hence the scalar product $\langle V u , u \rangle$ is well defined when $ u \in \dot{H}^1$. From the continuous embeddings we also have
	\begin{equation*}
		\|V\|_{\Dot{H}^1\rightarrow \Dot{H}^{-1}}\leq \|V\|_{ n/2,1}.
	\end{equation*} 
	We recall the following property. 
	
	\begin{lemma}
		Let $V \in L^{n/2, 1}$ and $\delta \ll1$ then there exist $K, W \in L^{n/2, 1}$ such that $K$ is a simple function and $\|W\|_{n /2, 1}\leq \delta$. 
	\end{lemma}
	We recall that simple functions are finite linear combinations of characteristic functions of sets of finite measure.
	\begin{proof}
		The statement follows directly from the density of simple functions in $L^{n/2, 1}$, see Theorem 1.4.13 in \cite{grafakos}. 
	\end{proof}

	From the previous lemma, we can decompose the potential $V$ as 
	\begin{equation*}
		V = W+K \ \quad \|W \|_{n/2, 1} \ll 1. 
	\end{equation*}
	Instead of looking for the solution of $(-\Delta +V)\psi=0$, in this section we will study $\psi \in \dot{H}^1$ as solution of the equation 
	\begin{equation*}
		(-\Delta + W) \psi = -K \psi.
	\end{equation*}
	We now consider the operator 
	\begin{equation*}
		(-\Delta + W) : \dot{H}^1 \to\dot{H}^{-1}
	\end{equation*} 
	where $-\D: \dot{H}^1 \to\dot{H}^{-1}$ is an isometry and 
	
	\begin{equation*}
		\|W\|_{\Dot{H}^1\rightarrow \Dot{H}^{-1}}\leq \|W\|_{ n/2,1} \ll 1. 
	\end{equation*}

	\begin{rem}
		\label{rem: psi in L^6,infty 1}

		The operator $-\D + W$ maps $\dot{H}^1$ into $\Dot{H}^{-1}$ even under the weaker assumption $W \in L^{n/2, 2}$, thanks to the embedding \eqref{eq: dot H^1 in lorentz}. However, the space $L^{n/2, 1}$ has the interest of being the dual of $L^{\frac{n}{n-2},\infty}$, this will allow us to integrate $W$ against the kernel of $(-\D)^{-1}$ and to state that $\int V \psi$, limit of $|x|^{n-2}\psi$ at infinity, is finite. 
		
	\end{rem}

	The smallness of $W$ allows us to invert the operator via a Neumann series. Indeed, writing 
	\begin{equation*}
		(-\Delta + W)^{-1} =(-\Delta)^{-1}( I +W(-\Delta)^{-1})^{-1}
	\end{equation*}
	we obtain a small perturbation of the identity and hence we can write, at least formally, 
	\begin{equation}
		\label{eq: formal series resolvent}
		(-\Delta + W)^{-1} =\sum_{j\geq 0} (-1)^j (-\Delta)^{-1} (W(-\Delta)^{-1})^j
	\end{equation}
	where the series is convergent thanks to the smallness of $W$.

	We then apply an idea from \cite{pinchover}: to construct the Green function we define the integral kernels corresponding to the operators in the series \eqref{eq: formal series resolvent}. 
	We do so by recurrence, setting
	\begin{align*}
		G_0(x,y)= c_n\frac{1}{ |x-y|^{n-2}},\ G_j(x,y)= c_n\int \frac{1}{|x-z|^{n-2}} W(z) G_{j-1}(z,y) dz
	\end{align*}
	where $c_n= n(n-2)|B(0,1)|$ is a constant depending on the dimension. 
	Here $G_0$ is the kernel of $(-\Delta)^{-1}$, $G_1$ that is given by 
	\begin{equation*}
		G_1(x,y)= c_n \int \frac{1}{|x-z|^{n-2}} W(z)c_n \frac{1}{|z-y|^{n-2}} dz
	\end{equation*}
	is the kernel of $(-\Delta)^{-1}W(-\Delta)^{-1}$ and so on, $G_j$ will be the kernel of $(-\Delta)^{-1}(W(-\Delta)^{-1})^j$. 
	
	To bound the integrals defining $G_j$ we first remark a useful inequality. 
	\begin{lemma}
		\label{lemma: bound integral gWg}
		Let $a(x)= \frac{1}{|x|^{n-2}}$. For $x\neq y$ it holds
		\begin{equation*}
			\int \frac{1}{|x-z|^{n-2}} |W(z)| \frac{1}{ |z-y|^{n-2}} dz\leq\|a\|_{\frac{n}{n-2}, \infty}\|W\|_{n/2,1} \frac{2^{n-1}}{|x-y|^{n-2} }.
		\end{equation*}
	\end{lemma}
	
	\begin{proof}
		We split the integral
		into the regions $\{|z-y|\leq \frac{|x-y|}{2}\}$ and $\{|z-y|> \frac{|x-y|}{2}\}$ so that $z\in B(y, \frac{|x-y|}{2})$ implies $|x-z|\geq \frac{|x-y|}{2}$ and we have 
		\begin{align*}
			\int \frac{1}{|x-z|^{n-2}} |W(z)| \frac{1}{ |z-y|^{n-2}} dz \leq& \frac{2^{n-2}}{|x-y|^{n-2}} \int |W(z)| \frac{1}{ |z-y|^{n-2}} dz\\
			&+ \frac{2^{n-2}}{|x-y|^{n-2}}\int \frac{1}{|x-z|^{n-2}} |W(z)| dz\\
			\leq & \frac{2^{n-2}}{|x-y|^{n-2}} \|W\|_{n/2,1} (\| |\cdot -y|^{-(n-2)}\|_{\frac{n}{n-2}, \infty} \\
			&\hspace{3.7cm}+ |\cdot -x|^{-(n-2)} \|_{\frac{n}{n-2}, \infty} )\\
			\leq & \|a\|_{\frac{n}{n-2}, \infty}\|W\|_{n/2,1} \frac{2^{n-1}}{|x-y|^{n-2} }.
		\end{align*}
		The last inequality follows from the fact that $L^{p, q}$ quasinorm are invariant by translation. 
		%of $g$ is invariant by translation (and actually this holds true for every $L^{p,q}$ quasinorm). Indeed, the superlevel of $g$ is a ball centered at 0 and of radius $\frac{1}{t}$ and since the Lebesgue measure of a ball is invariant under translation the Lorentz quasinorm of $g$ is unchanged.
	\end{proof}

\begin{rem}
	In the previous lemma we stated some useful properties for the kernel of $(-\D)^{-1}$ in $\R^n, n\geq 3$ when integrated against a function in $L^{n/2,1}$.	The two dimensional case, with its logarithmic behaviour in the kernel of $(-\D)^{-1}$ needs a dedicated approach.
\end{rem}
	
	\begin{theo}
		\label{th: G green function of -Delta +W}
		Let $W\in L^{n/2, 1}$ with $\|W\|_{n/2, 1}\ll1$ sufficiently small, then 
		\begin{equation}
			\label{eq: series definition G}
			G(x,y):= \sum_{j\geq 0} (-1)^j G_j(x,y)
		\end{equation}
		is the Green function of $-\Delta +W$ and is such that $|G(x,y)|\lesssim \frac{1}{|x-y|^{n-2}}$.
	\end{theo}
	\begin{rem}
		The theorem gives us a pointwise bound on the integral kernel of $(-\Delta +W)^{-1}$ by the integral kernel of $(-\D)^{-1}$. We deduce that $(-\Delta +W)^{-1}$ inherits any $L^p\to L^q$ or $L^{p,q}\to L^{p',q'}$ bound that $(-\D)^{-1}$ enjoys. 
	\end{rem}

	\begin{proof}[Proof of Theorem \ref{th: G green function of -Delta +W}]
		
		By Lemma \eqref{lemma: bound integral gWg} it is straightforward to bound $G_1$ by
		\begin{align*}
			|G_1(x,y)| \leq c_n^2 \|a\|_{\frac{n}{n-2}, \infty}\|W\|_{n/2,1} \frac{2^{n-1}}{|x-y|^{n-2} }=2^{n-1}c_n\|a\|_{\frac{n}{n-2}, \infty}\|W\|_{n/2,1} G_0(x,y)
		\end{align*}
		and setting $C=2^{n-1}c_n\|a\|_{\frac{n}{n-2}, \infty}\|W\|_{n/2,1}$ we obtain by induction 
		\begin{equation*}
			|G_j(x,y)|\leq C^j G_0(x,y).
		\end{equation*}
		Indeed, assuming $|G_{j-1}(x,y)|\leq C^{j-1} G_0(x,y)$ and applying again Lemma \ref{lemma: bound integral gWg} we directly obtain 
		\begin{align*}
			|G_j(x,y)|\leq & c_n^2 \int \frac{1}{|x-z|^{n-2}} |W(z)| \frac{C^{j-1}}{|z-y|^{n-2}} dz \\
			\leq &C^{j-1} c_n^2  \|a\|_{\frac{n}{n-2}, \infty}\|W\|_{n/2,1} \frac{2^{n-1}}{ |x-y|^{n-2}}= C^j G_0(x,y). 
		\end{align*}
		
		The constant $C$ is less than one thanks to the smallness of $W$ and hence the series \eqref{eq: series definition G} is convergent. The bound on $G$ follows directly from the one on $G_j$. 
		
		Finally, we check that $G$ is indeed the kernel of $(-\Delta + W)^{-1}$. Let  $\varphi, \psi\in C^{\infty}_0$ two test functions and $\langle\ ,\ \rangle$ the scalar product of $L^2$ %which is the duality coupling of $\dot{H}^1$
		\begin{align}
			\langle \psi, (-\Delta + W)^{-1} \varphi \rangle =& \langle \psi, \sum_{j\geq 0} (-1)^j (-\Delta)^{-1} (W(-\Delta)^{-1})^j \varphi\rangle \nonumber\\
			=& \sum_{j\geq 0}(-1)^j \langle \psi, (-\Delta)^{-1} (W(-\Delta)^{-1})^j \varphi \rangle \label{use_conv_series_on_sobolev}\\
			=& \sum_{j\geq 0 }(-1)^j  \int \psi(x) G_j(x,y) \varphi(y)\ dxdy \nonumber\\
			=& \int \psi(x) \sum_{j\geq 0 }(-1)^j G_j(x,y) \varphi(y)\ dxdy \label{use_series_is_abs_conv}\\
			=& \int \psi(x) G(x,y) \varphi(y)\ dxdy \nonumber
		\end{align}
		where to obtain \eqref{use_conv_series_on_sobolev} we used the fact that the series \eqref{eq: formal series resolvent} is convergent with respect to the topology of $\mathcal{B}(\Dot{H}^{-1}, \Dot{H}^{1})$ and for \eqref{use_series_is_abs_conv} we used the absolute convergence of the series \eqref{eq: series definition G}. 
		
	\end{proof}
	
	\section{Properties of a zero resonant state}
	\label{section: properties of zero resonance}
	
	The aim of this section is to prove Theorems \ref{th: res in weak sp} and \ref{th: properties of 0 resonance}. Using the decomposition $ V=W+K$ we can write the resonance $\psi$ as solution of $(-\Delta + W) \psi= -K\psi$ and by the Green function defined in Theorem \ref{th: G green function of -Delta +W} this solution is given by
	\begin{equation}
		\label{eq: psi is int of GK}
		\psi(x)= - \int G(x,y) K(y) \psi(y) dy.
	\end{equation}
	
	First of all we prove Theorem \ref{th: res in weak sp}, therefore obtaining that a resonant state $\psi$  is actually an eigenfunction, when $n\geq 5$, or is in a weak Lebesgue space, when $n=3,4$.

	\begin{proof}[\textbf{Proof of Theorem \ref{th: res in weak sp}}]

		$\bullet$ \textbf{Case $n \geq 5$.}
		We consider $\psi \in \dot{H}^1$ solution of the eigenvalue equation $(-\D + V)\psi =0$, or equivalently $(-\Delta + W) \psi= -K\psi$. We want to prove that $\psi$ is in $L^2$. We know that $\psi $ is a solution of the equation 
		\begin{equation}
			\label{eq: psi with invertible op}
			(I +(-\Delta)^{-1} W) \psi= -(-\Delta)^{-1}K\psi. 
		\end{equation}
		Thanks to Hardy-Littlewood-Sobolev and H\"older inequalities we have 
		\begin{equation*}
			(-\D)^{-1}:L^{\frac{2n}{n+4},2} \to L^2,\ \quad W : L^2 \to L^{\frac{2n}{n+4},2} 
		\end{equation*}
		where $\|W\|_{L^2 \to L^{\frac{2n}{n+4},2} }  \lesssim \|W\|_{n/2, 1}\ll 1$, hence the operator $(I +(-\Delta)^{-1} W) $ can be inverted via a Neumann series and its inverse defines an operator from $L^2$ to $L^2$. Applying this operator to \eqref{eq: psi with invertible op} we find 
		\begin{equation*}
			\psi = -(I + (-\D)^{-1}W)^{-1} (-\D)^{-1}K \psi
		\end{equation*}
		with 
		\begin{equation*}
			(I + (-\D)^{-1}W)^{-1} : L^2 \to L^2.
		\end{equation*}
		So $\psi $ is in $L^2$ if $(-\D)^{-1}K \psi$ is. Since $ n \geq 5$ we have the Sobolev embedding
		\begin{equation*}
			\Dot{H}^{2} \hookrightarrow L^{\frac{2n}{n-4},2} 
		\end{equation*}
		and its dual
		\begin{equation*}
			L^{\frac{2n}{n+4},2} \hookrightarrow \dot{H}^{-2}.
		\end{equation*}
		Now, $K$ is a simple function, therefore it belongs to the space $L^{n/3, \infty}$ and by H\"older inequality
		\begin{align*}
			\|(-\D)^{-1}K \psi\|_2 = \|K \psi\|_{\dot{H}^{-2}} \lesssim\|K \psi \|_{\frac{2n}{n+4},2} \lesssim \|K\|_{n/3, \infty} \|\psi \|_{\frac{2n}{n-2},2} < \infty
		\end{align*}
		thanks to the inclusion $	\Dot{H}^{1} \hookrightarrow L^{\frac{2n}{n-2},2}$. We have obtained $(-\D)^{-1}K \psi \in L^2$
		and hence $\psi \in L^2$.

		$\bullet$ \textbf{Case $n = 3,4$.} We recall from Theorem \ref{th: G green function of -Delta +W} that $|G(x,y)|\lesssim \frac{1}{|x-y|^{n-2}}$, then using \eqref{eq: psi is int of GK} for large enough $|x|$
		\begin{equation*}
			|x|^{n-2}| \psi(x)| \lesssim |x|^{n-2} \int \frac{1}{|x-y|^{n-2}}| K(y) \psi(y)| dy \lesssim \int| K(y) \psi(y)| dy <\infty.
		\end{equation*}
		where the last integral is finite. Indeed, $K$ is a simple function, so  $K   \in L^{p,q}$ for any $p$ and $q$. In particular $K \in L^{\frac{2n}{n+ 2}, 1} = (L^{\frac{2n}{n-2}, \infty})^*$ and $\psi \in L^{\frac{2n}{n-2}, \infty}$.

		We can then obtain $\psi\in L^{\frac{n}{n-2},\infty}$. Let $\chi$ be a smooth cutoff which is equal to 1 on a large enough compact set, for the compact part it holds $\psi\chi \in L^{\frac{n}{n-2},\infty}$ since $\chi, \psi  \in L^{\frac{2n}{n-2},\infty}$. For the part at infinity we can bound $\psi(1- \chi)$ by $\frac{1}{|x|^{n-2}}\in L^{\frac{n}{n-2},\infty}$. So applying \eqref{eq: triangular ineq lorentz sp} we conclude 
		\begin{equation*}
			\|\psi\|_{\frac{n}{n-2}, \infty} \lesssim\|\chi \psi\|_{\frac{n}{n-2}, \infty} + \|(1-\chi) \psi\|_{\frac{n}{n-2}, \infty}<\infty .
		\end{equation*}

		Now to determine the value of the limit we need to study the behavior of $|x|^{n-2}G(x,y)$ for large $|x|$ and $y$ that ranges in a compact set (the support of $K$). Using the second resolvent identity 
		\begin{equation}
			\label{comp: resolvent id 2}
			(-\Delta + W)^{-1}= (-\Delta)^{-1} - (-\Delta)^{-1}W(-\Delta + W)^{-1}
		\end{equation}
		we can write 
		\begin{equation}
			\label{eq: |x|G(x,y)}
			|x|^{n-2}G(x,y)= c_n\frac{|x|^{n-2}}{ |x-y|^{n-2}} - c_n\int \frac{|x|^{n-2}}{4\pi |x-z|^{n-2}} W(z) G(z,y) dz .
		\end{equation}
		The first term in $|x|^{n-2}G(x,y)$ converges to $c_n$, for the second term we split the integral in the regions $B(0, |x|/2)$ and its complementary. First, taking $|x|$ large enough if $z\in B(0, |x|/2)^c$ then we will have 
		\begin{align*}
			|z-y|\geq ||z|-|y|| = |z|-|y| \geq \frac{|x|}{2} - |y| >0.
		\end{align*}
		Using this bound we obtain 
		\begin{align}
			\int_{B(0, |x|/2)^c} \frac{|x|^{n-2}}{4\pi |x-z|^{n-2}}| W(z) G(z,y)| dz\lesssim & \int_{B(0, |x|/2)^c} \frac{1}{ |x-z|^{n-2}} |W(z)| \frac{|x|^{n-2}}{(\frac{|x|}{2} - |y|)^{n-2} } dz \nonumber\\
			\lesssim &\int_{B(0, |x|/2)^c} \frac{1}{ |x-z|^{n-2}} |W(z)| dz\nonumber\\
			\lesssim & \|a\|_{\frac{n}{n-2}, \infty} \|W \mathbbm{1}_{B(0, |x|/2)^c}\|_{n/2,1} \xrightarrow[]{|x|\to +\infty} 0 \label{comp: limit |x|G outside ball}.
		\end{align}
		The norm $\|W \mathbbm{1}_{B(0, |x|/2)^c}\|_{n/2,1}$ is defined by an integral of the distribution function. Then the convergence to zero is due to the fact that the superlevel of $W \mathbbm{1}_{B(0, |x|/2)^c}$ tends to the empty set as $|x|\to +\infty$ and we can pass the limit in the integral thanks to the domination 
		\begin{equation*}
			d_{W\mathbbm{1}_{B(0, |x|/2)^c}}(t)^{2/n} \leq d_W(t)^{2/n} \in L^1(\mathbb{R}^+) .
		\end{equation*}
		On the other hand, for $|x|\to +\infty$ we have the pointwise convergence of 
		\begin{equation*}
			\mathbbm{1}_{B(0, |x|/2)}(z)c_n\frac{|x|^{n-2}}{|x-z|^{n-2}} W(z) G(z,y) \to c_n W(z) G(z,y)
		\end{equation*}
		and since the points $z \in B(0, \frac{|x|}{2})$ satisfy $|x-z|> \frac{|x|}{2}$ we have the domination 
		\begin{equation*}
			\mathbbm{1}_{B(0, |x|/2)}(z)\frac{|x|^{n-2}}{4\pi |x-z|^{n-2}} |W(z) G(z,y)| \lesssim |W(z)| \frac{1}{|z-y|^{n-2}} \in L^1.
		\end{equation*}
		So again by dominated convergence 
		\begin{equation}
			c_n \int_{B(0, |x|/2)} \frac{|x|^{n-2}}{4\pi |x-z|^{n-2}} W(z) G(z,y) \xrightarrow[]{|x|\to +\infty}\ c_n\int  W(z) G(z,y) dz . \label{comp: limit |x|G inside ball}
		\end{equation}
		Summing together \eqref{comp: limit |x|G outside ball} and \eqref{comp: limit |x|G inside ball} in \eqref{eq: |x|G(x,y)} we obtain 
		\begin{equation*}
			\lim_{|x|\to +\infty}|x|^{n-2}G(x,y) = c_n- c_n\int W(z) G(z,y) dz 
		\end{equation*}
		and since $|x|^{n-2}G(x,y)\lesssim 1$ and $K\psi \in L^1$ we can pass the limit in the integral in \eqref{eq: psi is int of GK}. This yields
		\begin{align*}
			\lim_{|x|\to +\infty}|x|^{n-2}\psi(x)=& -c_n\int K(y) \psi(y)dy + c_n\int \left(\int W(z) G(z,y) dz  \right)  K(y) \psi(y)dy .
		\end{align*}
		Now exchanging the order of integration in the second term and using relation \eqref{eq: psi is int of GK} we have
		\begin{align*}
			\int \left(\int W(z) G(z,y) dz  \right)  K(y) \psi(y)dy = &\int W(z) \int G(z,y) K(y) \psi(y)dy dz\\
			=& - \int W(z) \psi(z) dz
		\end{align*}
		from which the statement follows since $V= W +K$.

	\end{proof}
	
	Thanks to the previous proof we can derive the behavior of $\psi$ at infinity. We can prove further decay for $\psi$, and therefore better integrability, under suitable orthogonality assumptions on $\psi$. To obtain decay for $\psi$ we will prove that $|x|^\alpha \psi$ is bounded at infinity for a suitable $\alpha$.
	
	We take $R>0$ sufficiently large such that
	\begin{equation*}
		supp \ K \subset B(0,R) \quad \textnormal{and } \quad 	|x|^{n-2}|\psi(x)| \in L^\infty(B(0,R)^c). 
	\end{equation*}
	We recall that such $R$ exists thanks to the fact that $K$ is a simple function together with the fact that $|x|^{n-2}|\psi(x)| $ has a finite limit. We study the behavior of $\psi $ in $B(0,R)^c$. To do so we define the spaces 
	\begin{equation}
		\label{def: B_alpha}
		\mathcal{B}_\alpha = |x|^{-\alpha} L^\infty (B(0,R)^c) = \{ f\ :\ |x|^\alpha f \in L^\infty(B(0,R)^c)\ \}
	\end{equation}
	with the natural norm 
	\begin{equation*}
		\|f\|_{\mathcal{B}_\alpha } = \||x|^\alpha f\  \mathbbm{1}_{B(0,R)^c}\|_\infty. 
	\end{equation*}
	In Proposition \ref{prop: decay psi} we prove that $\psi $ belongs to such spaces $\mathcal{B}_\alpha$ for suitable $\alpha$. 
	
	We will use the following inequality. A proof can be found in Appendix \ref{app: proof of ineq}.

	\begin{lemma}
		\label{lemma: algebraic ineq}
		Let $N \in \N$ and $k =0, \ldots , N$ $l= 0, \ldots, k$. There exist $ c_{kl} \in \R$ such that 
		\begin{align*}
			\bigg| \frac{1}{|x-y|^{n-2}} - \frac{1}{|x|^{n-2}} \sum_{k=0}^N \sum_{\substack{l=0\\ k + l \leq N}}^k c_{kl}& \frac{|y|^{k+l}}{|x|^{k+l}} \cdot\frac{(x\cdot y )^{k-l}}{(|x||y|)^{k-l}}\bigg|\\
			&\lesssim  \frac{1}{|x-y|^{n-2}}  \left( \frac{|y|^{N + 1}}{|x|^{N + 1}} + \frac{|y|^{N + n-2}}{|x|^{N + n-2}}\right)
		\end{align*}
		for all $x, y \in \R^n \setminus \{0\}$ and $x \neq y $. 	
	\end{lemma}

	\begin{notation}
		For $ y \in \R^n$ and $\alpha  = (\alpha_1, \ldots, \alpha_n)\in \N^n$ we denote by $y^\alpha= \prod_{k=1}^n y_k^{\alpha_k}$.
	\end{notation}

	\begin{prop}
		\label{prop: decay psi}
		Let $\psi \in \dot{H}^1$ a solution of the equation $(-\D + V) \psi =0$ and $N = 0,1,2$. Then $ \psi \in \mathcal{B}_{N + n-1}$, that is $\psi$ decays like $|x|^{-(N + n-1)}$ at infinity, if $\int y^\alpha V \psi=0 $ for all multi-indices $ \alpha \in \N^n , |\alpha|\leq N$.
		
	\end{prop}
	
	\begin{proof}

		Using the integral kernel of $\Delta^{-1}$ on $\mathbb{R}^n$ we can write 
		\begin{equation}
			\label{eq: psi = delta V}
			\psi(x)= -c_n \int \frac{1}{ |x-y|^{n-2}} V(y) \psi (y) dy .
		\end{equation}
		By assumption for $N\leq 2$ and $ k =0, \ldots N$ we have 
		\begin{equation*}
			\int \bigg(\sum_{\substack{l=0\\ k + l \leq N}}^k c_{kl}\frac{|y|^{k+l}}{|x|^{k+l}} \cdot\frac{(x\cdot y )^{k-l}}{(|x||y|)^{k-l}} \bigg) V(y) \psi (y)dy =0. 
		\end{equation*}
		Indeed, the terms in the sum are given by 
		\begin{equation*}
			|y|^{k+l}\frac{(x\cdot y )^{k-l} }{|y|^{k-l}} = |y|^{2l} (x\cdot y )^{k-l} 
		\end{equation*}
		and are therefore polynomials in $y$ of degree $ k + l $, since $k + l \leq N$ they vanish when we integrate them against $V\psi$. 
		
		Thanks to these vanishing quantities, we then rewrite \eqref{eq: psi = delta V} as 
		\begin{align}
			\psi(x)=& c_n \int\bigg( \frac{1}{|x|^{n-2}} 
			\bigg(\sum_{k=0}^N\sum_{\substack{l=0\\ k + l \leq N}}^k c_{kl}\frac{|y|^{k+l}}{|x|^{k+l}} \cdot\frac{(x\cdot y )^{k-l}}{(|x||y|)^{k-l}} \bigg)- \frac{1}{ |x-y|^{n-2}}\bigg) V(y) \psi (y) dy \nonumber \\
			=& \underbrace{c_n  \int_{|y|\leq R}\bigg( \frac{1}{|x|^{n-2}} 
				\bigg(\sum_{k=0}^N\sum_{\substack{l=0\\ k + l \leq N}}^k c_{kl}\frac{|y|^{k+l}}{|x|^{k+l}} \cdot\frac{(x\cdot y )^{k-l}}{(|x||y|)^{k-l}} \bigg)- \frac{1}{ |x-y|^{n-2}}\bigg)V(y) \psi (y) dy}_{=: f(x)} \label{def: f}\\
			& + c_n\int_{|y|\geq R}\bigg( \frac{1}{|x|^{n-2}} 
			\bigg(\sum_{k=0}^N\sum_{\substack{l=0\\ k + l \leq N}}^k c_{kl}\frac{|y|^{k+l}}{|x|^{k+l}} \cdot\frac{(x\cdot y )^{k-l}}{(|x||y|)^{k-l}} \bigg)- \frac{1}{ |x-y|^{n-2}}\bigg) W(y) \psi (y) dy \nonumber
		\end{align}
		where we recall that $R$ is a sufficiently large radius such that
		\begin{equation*}
			supp \ K \subset B(0,R), \quad \ |x|^{n-2}|\psi(x)| \in L^\infty(B(0,R)^c). 
		\end{equation*}
		Defining the operator
		\begin{equation}
			\label{def: S}
			\mathcal{S} : \varphi \mapsto c_n \int_{|y| \geq R} \bigg( \frac{1}{|x|^{n-2}} 
			\bigg(\sum_{k=0}^N\sum_{\substack{l=0\\ k + l \leq N}}^k c_{kl}\frac{|y|^{k+l}}{|x|^{k+l}} \cdot\frac{(x\cdot y )^{k-l}}{(|x||y|)^{k-l}} \bigg)- \frac{1}{ |x-y|^{n-2}}\bigg) W (y) \varphi (y) dy
		\end{equation}
		we can rewrite the previous identity as 
		\begin{equation}
			\label{eq: psi = f + W psi}
			\psi = f + \mathcal{S} \psi .
		\end{equation}
		We will prove that the operator $\mathcal{S}$ is a contraction in the spaces $\mathcal{B}_{N+n-1}, \mathcal{B}_{N+n-2}$ and that $f \in \mathcal{B}_{N+n-1}\subset \mathcal{B}_{N+n-2}$, hence the equation 
		\begin{equation}
			\label{eq: fixed point W}
			\varphi = f + \mathcal{S} \varphi
		\end{equation}
		has a unique solution in $\mathcal{B}_{N+n-1}$ as well as a unique solution in $\mathcal{B}_{N+n-2}$. This proves the statement. Indeed, let $N=0$, then identity \eqref{eq: psi = f + W psi} together with item $ii)$ in Theorem \ref{th: res in weak sp} tell us that $\psi$ is the unique solution in $\mathcal{B}_{n-2}$ of the fixed point problem \eqref{eq: fixed point W}. The problem \eqref{eq: fixed point W} has a unique solution in $\mathcal{B}_{N + n-1}$ as well, since 
		\begin{equation*}
			\mathcal{B}_{N + n-1}\subset  \mathcal{B}_{n-2} \quad N=  0, 1,2
		\end{equation*}
		then $\psi $ must be in $\mathcal{B}_{N + n-1}$. 
		
		To conclude we must prove $f \in \mathcal{B}_{N+n-1}\subset \mathcal{B}_{N+n-2}$ and that $\mathcal{S}$ is a contraction on $\mathcal{B}_{N+n-1}$ and $\mathcal{B}_{N+n-2}$. 
		
		We use Lemma \eqref{lemma: algebraic ineq} to prove $f \in \mathcal{B}_{N+n-1}\subset \mathcal{B}_{N+n-2}$. Let $|x|\geq 2R$, from the definition of $f$ we first have 
		\begin{align*}
			|f(x)|\lesssim & \int_{|y| \leq R}  \frac{1}{|x-y|^{n-2}}  \left( \frac{|y|^{N + 1}}{|x|^{N + 1}} + \frac{|y|^{N + n-2}}{|x|^{N + n-2}}\right) |V(y) \psi (y) |dy \\
			\lesssim & \frac{1}{(|x|- R)^{n-2}} \left( \frac{1}{|x|^{N + 1}} + \frac{1}{|x|^{N + n-2}}\right) \int_{|y| \leq R} |V| |\psi| dy 
		\end{align*}
		hence $|x|^{N+n-1} |f(x)|$ is bounded for $|x|\geq 2R$. For $R \leq |x|\leq 2R$ we rewrite $f$ as 
		\begin{align*}
			f(x)= & c_n  \int_{|y|\leq R} \frac{1}{|x|^{n-2}} 
			\bigg(\sum_{k=0}^N\sum_{\substack{l=0\\ k + l \leq N}}^k c_{kl}\frac{|y|^{k+l}}{|x|^{k+l}} \cdot\frac{(x\cdot y )^{k-l}}{(|x||y|)^{k-l}} \bigg)V(y) \psi (y) dy\\ &
			+ \psi (x)  + \int_{|y|\geq R} \frac{1}{|x-y|^{n-2}} V (y) \psi (y) dy  
		\end{align*}
		then since $ |x |^{n-2} \psi (x) \in L^\infty (B(0, R)^c)$, the bound
		\begin{align*}
			|f(x)| \lesssim & \int  |V(y) \psi (y)| dy  + |\psi (x)| +  \int_{|y|\geq R} \frac{|V (y)|}{|x-y|^{n-2}} | y |^{n-2} |\psi (y)| dy 
		\end{align*}
		implies that $|x|^{N+n-1} |f(x)|$ is bounded for $R \leq |x|\leq 2R$ . 
		
		We now prove that $\mathcal{S}$ is a contraction on $\mathcal{B}_{N+n-2}$. 
		Thanks to Lemma \ref{lemma: algebraic ineq}
		\begin{align}
			|\mathcal{S} \varphi(x)|\lesssim & \int_{|y|\geq R} \frac{1}{|x-y|^{n-2}}  \left( \frac{|y|^{N + 1}}{|x|^{N + 1}} + \frac{|y|^{N + n-2}}{|x|^{N + n-2}}\right) |W (y)\varphi(y)| dy \nonumber \\
			\lesssim & \int_{|y|\geq R} \frac{1}{|x-y|^{n-2}} \frac{|y|^{N + 1}}{|x|^{N + 1}}|W (y)\varphi(y)| dy  \nonumber	\\ 
			& + \frac{1}{|x|^{N + n-2}}\|\varphi\|_{\mathcal{B}_{N+n-2}}\int  \frac{|W (y)|}{|x-y|^{n-2}}  dy \nonumber \\
			\lesssim & \int_{|y|\geq R} \frac{1}{|x-y|^{n-2}} \frac{|y|^{N + 1}}{|x|^{N + 1}}|W (y)\varphi(y)| dy  \label{comp: int y^N+1 / x^N+1}	\\ 
			& + \frac{1}{|x|^{N + n-2}}\|\varphi\|_{\mathcal{B}_{N+n-2}} \|W\|_{n/2, 1}  \label{comp: S phi less} \nonumber
		\end{align}
		We split the domain of integration in the regions 
		\begin{eqnarray}
			\{ |y| \geq R \} & = & \left\{ |y| \geq R \ \mbox{and} \ |y| < \frac{|x|}{2} \right\} \sqcup  \left\{ |y| \geq R \ \mbox{and} \  \frac{|x|}{2} \leq  |y| \leq 2 |x| \right\}\nonumber\\
			& & \sqcup  \left\{ |y| \geq R \ \mbox{and} \ |y| > 2 |x| \right\}  \nonumber \\
			& = : & E_{<} \sqcup E_\approx \sqcup E_>  \label{def: regions E}
		\end{eqnarray} 
		and we remark the following properties 
		\begin{align*}
			|y|^{-1} \lesssim |x|^{-1} \ \textnormal{on } E_> \cup E_\approx \quad \quad \quad  |x|^{-1} \lesssim |y|^{-1},\ |x-y|^{-1} \lesssim |x|^{-1} \ \textnormal{on } E_<.
		\end{align*}
		We use the inequalities 
		\begin{align*}
			\frac{|y|^{N + 1}}{|x-y|^{n-2}|x|^{N + 1}}|\varphi(y)| \lesssim & \begin{cases}
				\dfrac{|y|^{N+ n-2} |\varphi(y)|}{|x|^{N+ n-2}|x-y|^{n-2}} &  \textnormal{on } E_>\cup E_\approx\\
				\dfrac{|y|^{N+ n-2} |\varphi(y)|}{|x|^{N} |y|^{n-2}|x-y|^{n-2}}  &  \textnormal{on } E_<
			\end{cases}\\
			\lesssim & \begin{cases}
				\dfrac{1}{|x|^{N+ n-2}}\|\varphi\|_{\mathcal{B}_{N+n-2}}\dfrac{1}{|x-y|^{n-2}}  & \textnormal{on } E_>\cup E_\approx\\
				\dfrac{1}{|x|^{N}}\|\varphi\|_{\mathcal{B}_{N+n-2}}\dfrac{1}{|y|^{n-2}|x-y|^{n-2}} &  \textnormal{on } E_<
			\end{cases}
		\end{align*}
		obtaining
		\begin{align*}
			\int_{|y|\geq R} \frac{1}{|x-y|^{n-2}} \frac{|y|^{N + 1}}{|x|^{N + 1}}|W (y)\varphi(y)| dy  \lesssim  & \|\varphi\|_{\mathcal{B}_{N+n-2}} \bigg (\dfrac{1}{|x|^{N+ n-2}}\int_{ E_>\cup E_\approx}  \dfrac{|W(y)|}{|x-y|^{n-2}} dy \\
			& \hspace{2.3cm} + \dfrac{1}{|x|^{N}}\int_{ E_<}  \dfrac{|W(y)|}{|y|^{n-2}|x-y|^{n-2}} dy  \bigg)\\
			\lesssim & \dfrac{1}{|x|^{N+ n-2}}\|\varphi\|_{\mathcal{B}_{N+n-2}} \|W\|_{n/2, 1} 
		\end{align*}
		where we used Lemma \ref{lemma: bound integral gWg} in the last inequality.

		Having estimated the integral in \eqref{comp: int y^N+1 / x^N+1} we can go back to the inequality on $|\mathcal{S} \varphi(x)|$ from which we finally have 
		\begin{align*}
			|\mathcal{S} \varphi(x)|\lesssim  \frac{1}{|x|^{N+ n-2}}\|\varphi\|_{\mathcal{B}_{N+n-2}}\|W\|_{n/2, 1} 
		\end{align*}
		hence 
		\begin{equation*}
			\|\mathcal{S} \varphi\|_{\mathcal{B}_{N+n-2}} \lesssim \|W\|_{n/2, 1} \|\varphi\|_{\mathcal{B}_{N+n-2}}
		\end{equation*}
		which implies that $\mathcal{S}$ is a contraction on $ \mathcal{B}_{N+n-2}$ thanks to the smallness of $ \|W\|_{n/2, 1}$. 
		
		As a last step, we need to prove that $\mathcal{S}$ is a contraction on $ \mathcal{B}_{N+n-1}$ as well. We still use the subdivision \eqref{def: regions E}. As we found before, thanks to Lemma \ref{lemma: algebraic ineq}
		\begin{align*}
			|\mathcal{S} \varphi(x)|\lesssim & \int_{|y|\geq R} \frac{1}{|x-y|^{n-2}}  \left( \frac{|y|^{N + 1}}{|x|^{N + 1}} + \frac{|y|^{N + n-2}}{|x|^{N + n-2}}\right) |W (y)\varphi(y)| dy . 
		\end{align*}
		We use the inequalities
		\begin{align*}
			\frac{|y|^{N + 1}}{|x-y|^{n-2}|x|^{N + 1}}|\varphi(y)| =& \frac{|y|^{N +n- 1}}{|x-y|^{n-2}|x|^{N + 1}|y|^{n-2}}|\varphi(y)|\\
			\lesssim & 
			\begin{cases}
				\dfrac{\|\varphi\|_{\mathcal{B}_{N+n-1}}}{|x|^{N+n-1
				}} \cdot \dfrac{1}{|x-y|^{n-2}} & \textnormal{on } E_>\cup E_\approx\\
				\dfrac{\|\varphi\|_{\mathcal{B}_{N+n-1}}}{|x|^{N+n-1}} \cdot \dfrac{1}{|y|^{n-2}} & \textnormal{on } E_<
			\end{cases}
		\end{align*}
		and
		\begin{align*}
			\frac{|y|^{N + n-2}}{|x-y|^{n-2}|x|^{N + n-2}} |\varphi(y)| =& \frac{1}{|x-y|^{n-2}} \cdot	\frac{|y|^{N + n-1}}{|x|^{N + n-2}|y|} |\varphi(y)|\\
			\lesssim& 
			\begin{cases}
				\dfrac{\|\varphi\|_{\mathcal{B}_{N+n-1}}}{|x|^{N + n-1}}\cdot \dfrac{1}{|x-y|^{n-2}} & \textnormal{on } E_>\cup E_\approx\\
				\dfrac{\|\varphi\|_{\mathcal{B}_{N+n-1}}}{|x|^{N +1}}\cdot \dfrac{1}{|y|^{n-2}|x-y|^{n-2}}    & \textnormal{on } E_<,
			\end{cases}
		\end{align*}
		thanks to which we can bound $	|\mathcal{S} \varphi(x)|$ by 
		\begin{align*}
			|\mathcal{S} \varphi(x)| \lesssim & \frac{\|\varphi\|_{\mathcal{B}_{N+n-1}}}{|x|^{N + n-1}} \left ( \int_{E_>\cup E_\approx} \frac{|W(y)|}{|x-y|^{n-2}} dy + \int_{E_<} \frac{|W(y)|}{|y|^{n-2}} dy\right) \\
			& + \frac{\|\varphi\|_{\mathcal{B}_{N+n-1}}}{|x|^{N +1}} \int_{E_<} \frac{|W(y)|}{|y|^{n-2}|x-y|^{n-2}}\\
			\lesssim & \frac{\|\varphi\|_{\mathcal{B}_{N+n-1}}}{|x|^{N + n-1}}  \|W\|_{n/2, 1}  
		\end{align*}
		where we used again Lemma \ref{lemma: bound integral gWg} to obtain the last inequality. The previous bound is equivalent to 
		\begin{equation*}
			\|\mathcal{S} \varphi\|_{\mathcal{B}_{N+n-1}} \lesssim \|W\|_{n/2, 1} \|\varphi\|_{\mathcal{B}_{N+n-1}}
		\end{equation*}
		which proves that  $\mathcal{S}$ is a contraction on $ \mathcal{B}_{N+n-1}$ thanks to the smallness of $ \|W\|_{n/2, 1}$.

	\end{proof}

	\begin{rem}
		\label{rem: decay psi implies orthogon}
		In the previous proof we showed that $\int V \psi =0$ implies $\psi \in \mathcal{B}_{n-1}$. We remark that also the opposite implication is true: if $\psi \in \mathcal{B}_{n-1}$ then the integral of $V \psi $ must be zero. The same holds for the integral of $V\psi$ with monomials of order one, indeed we show below that $\int y_k V \psi =0$ for all $k =1, \ldots, n$ if and only if $\psi \in \mathcal{B}_{n}$.
		
		Assume $ \psi \in \mathcal{B}_{n-1}$, since $ \psi $ is also a solution of $(-\D + V) \psi =0$ we have 
		\begin{equation*}
			\int V (y)\psi (y)= \int\D\psi(y) = \lim_{\varepsilon \to 0 } \int \D\psi (y)\chi (\varepsilon y)
		\end{equation*}
		with $\chi$ a smooth cutoff on $ B(0,1)$. After integration by parts we obtain 
		\begin{equation*}
			\int V(y) \psi(y) = \lim_{\varepsilon \to 0 } \int \varepsilon^2 \psi(y) \D\chi (\varepsilon y) =0 
		\end{equation*}
		where the limit is zero by dominated convergence, thanks to the domination 
		\begin{align*}
			|\varepsilon^2 \psi(y) \D\chi (\varepsilon y)| =  |\varepsilon y|^2 |\D\chi (\varepsilon y)| \frac{1}{|y|^2} | \psi(y)|\lesssim \frac{| \psi(y)|}{|y|^2} 
		\end{align*}
		with $\frac{| \psi(y)|}{|y|^2} \lesssim \frac{1}{|y|^{n + 1}} $ integrable on the support of $\D\chi$ (it would actually be sufficient that $\psi $ belongs to $\mathcal{B}_{n-2+ \delta}$ for some positive $\delta$). 
		
		Similarly, if $ \psi \in \mathcal{B}_{n}$ we have 
		\begin{align*}
			\int y_k V(y) \psi(y) = &\lim_{\varepsilon \to 0 } \int y_k  \D\psi (y)\chi (\varepsilon y) =  \lim_{\varepsilon \to 0 } \int (y_k \varepsilon^2  \D\chi (\varepsilon y) + 2 \varepsilon \partial_{y_k}\chi (\varepsilon y )) \psi (y) \\
			=&\  0
		\end{align*}
		again by dominated convergence since 
		\begin{align*}
			|(y_k \varepsilon^2  \D\chi (\varepsilon y) + 2 \varepsilon \partial_{y_k}\chi (\varepsilon y )) \psi (y)| = &  |\varepsilon^2 y_k  | y| \D\chi (\varepsilon y) + 2 | \varepsilon y| \partial_{y_k}\chi (\varepsilon y )| \frac{1}{|y|} |\psi (y)| \lesssim\frac{| \psi(y)|}{|y|} 
		\end{align*}
		with $\frac{|\psi (y)|}{|y|}\lesssim \frac{1}{|y|^{n+1}} $ integrable on the support of $\partial_{y_k}\chi$ (as before, $\psi \in \mathcal{B}_{n-1 + \delta}$ for $\delta>0$ would be enough). 
		
		We also remark that we can not use the same reasoning to prove that under the assumption $\psi \in \mathcal{B}_{n+1}$ we have $\int y^\alpha V \psi =0$ for all $|\alpha|=2$. Indeed, for $k \neq l $ 
		\begin{align*}
			\int y_k y_l  V(y) \psi(y) = &\lim_{\varepsilon \to 0 } \int y_ky_l   \D\psi (y)\chi (\varepsilon y) \\
			=&  \lim_{\varepsilon \to 0 } \int (y_k y_l\varepsilon^2  \D\chi (\varepsilon y) + 2 \varepsilon y_k  \partial_{y_l}\chi (\varepsilon y ) +2 \varepsilon y_l \partial_{y_k}\chi (\varepsilon y )) \psi (y) =0
		\end{align*}
		where the domination in the dominated convergence theorem is given by $\psi$ itself. On the contrary, integrating against $y_k^2$ we can not conclude that the integral vanishes since 
		\begin{align*}
			\int y_k^2  V(y) \psi(y) = &\lim_{\varepsilon \to 0 } \int y_k^2  \D\psi (y)\chi (\varepsilon y) \\
			=&  \lim_{\varepsilon \to 0 } \int (y_k^2 \varepsilon^2  \D\chi (\varepsilon y) + 4y_k  \varepsilon \partial_{y_k}\chi (\varepsilon y ) + 2 \chi (\varepsilon y )) \psi (y)\\
			= & 2\int  \psi (y) dy .
		\end{align*}
		However, we remark the following fact. If $\int y^\alpha V\psi =0$ for all $\alpha$ with $|\alpha|\leq 2$, the previous computations shows us that $\int  \psi =0$. We also know that under this orthogonality assumption $\psi \in \mathcal{B}_{n+1}$, as stated in Theorem \ref{th: properties of 0 resonance}, hence a sufficiently fast decaying eigenfunction must be orthogonal to constants. 
	\end{rem}
	\bigskip
	
	We are now able to prove Theorem \ref{th: properties of 0 resonance}. 
	
	\begin{proof}[\textbf{Proof of Theorem \ref{th: properties of 0 resonance}}]
		
		The proof is a direct consequence of the previous propositions. 
		
		We start by proving item $i)$. If $\int V \psi =0$, we can apply Proposition \ref{prop: decay psi} for $N=0$, obtaining $\psi  \in \mathcal{B}_{n-1}$ and consequently $ \psi \in L^2 (B(0,R)^c)$. Moreover, $\psi \in \dot{H}^1\subset L^{\frac{2n}{n-2}}$ by assumption and $L^{\frac{2n}{n-2}}_{loc} \subset L^2_{loc}$ since $\frac{2n}{n-2} \geq 2 $ for $n\geq 3$. For $n=3,4$, if $\int V\psi\neq 0$ the limit of $|x|^{n-2}\psi $ is finite and non zero thanks to Theorem \ref{th: res in weak sp}. Hence, $\frac{c}{|x|^{n-2}}\leq\psi \leq \frac{c'}{|x|^{n-2}}$ at infinity and therefore $\psi$ is not in $ L^2$. We conclude that for dimensions three and four $\int V\psi =0$ is also a necessary condition to have $\psi \in L^2$.

		Now assume $\int y_k V(y)\psi(y)dy=0$ for all $k = 1, \ldots , n$.
		We can apply Proposition \ref{prop: decay psi} with $N=1$ which implies $\psi \in \mathcal{B}_n$ and hence $ \psi\in L^{1, \infty}$, since $\frac{1}{|x|^n } \in L^{1, \infty}$. Item $ii)$ is then proved. 
		
		We conclude by proving item $iii)$. Thanks to the assumptions on $\int y_k V(y)\psi(y)dy$ and $\int y_k y_l V(y)\psi(y)dy$ for $k, l = 1, \ldots, n$ we can apply Proposition \ref{prop: decay psi} with $N=2$. Then
		$\psi \in \mathcal{B}_{n+1}$ implies $\psi \in L^1(B(0,R)^c)$ while $\psi\in L^2$, that we already have from item $i)$ since $\int V \psi=0$, implies $\psi \in L^1_{loc}$. Alternatively, this last inclusion can be proved by interpolation. 
		Since $\psi $ is in $\mathcal{B}_{n+1}$ we have that $|x|\psi$ is bounded by $|x|^{-n}$ at infinity and hence $|x|\psi \in L^{1, \infty}$. We can then write 
		\begin{equation*}
			\psi (x)	= \frac{1}{|x|} b(x)  \ \textnormal{ with }\ \frac{1}{|x|} \in L^{n, \infty},  b \in L^{1, \infty}, 
		\end{equation*}
		which by H\"older inequality \eqref{eq: holder in lorentz space} implies $\psi \in L^{\frac{n}{n+1}, \infty}$. Having 
		\begin{equation*}
			\psi \in L^{\frac{n}{n+1}, \infty} \cap L^{2, \infty}
		\end{equation*}
		we can conclude $\psi \in L^1$ by interpolation (Proposition \ref{prop: f in intersection of L^p infty}).

	\end{proof}

	\vspace{1cm}

	\thanks{\textbf{Acknowledgments:} This work received support from the University Research School EUR-MINT
		(State support by the National Research Agency' Future Investments program, reference number ANR-18-EURE-0023). The author would like to thank Professor H. Mizutani for useful discussions on this topic and the anonymous referee for his/her comments that allowed to improve the paper, in particular Lemma \ref{lemma: algebraic ineq}.}

	\appendix 
	
	\section{Facts about Lorentz spaces}
	\label{section: appendix lorentz spaces}
	We collect here a few properties of Lorentz spaces that we have used. The following statements hold on $\mathbb{R}^n$ for any $n$. 
	
	First of all we recall the definition of the quasinorm 
	\begin{align*}
		\|f\|_{p,q}:= p^{1/q}\left(\int_0^\infty t^{q-1} (d_f(t))^{q/p} dt \right)^{1/q} 
	\end{align*}
	for $q<\infty$ or 
	\begin{equation*}
		\| f\|_{p, \infty} := \sup_{t\geq 0} t d_f(t)^{1/p}<\infty
	\end{equation*}
	otherwise. Observe that $L^{p,p}= L^p$ and that $L^{p,\infty}$ is the weak $L^p$ space. The quantity we just defined is only a quasinorm, since it holds 
	\begin{equation}
		\label{eq: triangular ineq lorentz sp}
		\|f+g\|_{p,q} \lesssim_{p,q} \|f\|_{p,q} + \|g\|_{p,q}
	\end{equation}
	(see inequality (1.4.9) in \cite{grafakos}). Only for $p >1$ and any $q \in [1,\infty]$ the space $L^{p,q}$ is normable (\cite{hunt_lorentz_sp}). 
	
	\begin{rem}
		\label{rem: norms via decr rearr}
		We recall that the quasinorm of the Lorentz space can also be defined via the decreasing rearrangement $f^*(t)= \inf \{s>0\ |\ d_f(s)<t\}$ as 
		\begin{equation*}
			\|f\|_{p,q}:= \left(\int_0^\infty t^{q/p-1} f^*(t)^{q} dt \right)^{1/q} <\infty
		\end{equation*}
		for $q<\infty$ or 
		\begin{equation*}
			\| f\|_{p, \infty} := \sup_{t\geq 0} t^{1/p} f^*(t)<\infty
		\end{equation*}
		otherwise.

	\end{rem}
	
	Lorentz spaces are growing spaces with respect to the second index, in particular we have the chain of inclusions 
	\begin{equation}
		\label{eq: inclusions lorentz spaces}
		L^{p,q_1} \subset  L^{p,q_2}\ \textnormal{ for any } q_1<q_2
	\end{equation}
	(see Proposition 1.4.10 in \cite{grafakos}). 
	
	We also have a two indexed H\"{o}lder inequality \begin{equation}
		\label{eq: holder in lorentz space}
		\|f g\|_{L^{p,q}} \leq \|f\|_{L^{p_1,q_1}} \|g\|_{L^{p_2,q_2}}
	\end{equation}
	for any $p_1,q_1,p_2,q_2$ such that $\frac{1}{p_1}+ \frac{1}{p_2}=\frac{1}{p}$ and $\frac{1}{q_1}+ \frac{1}{q_2}=\frac{1}{q}$. Inequality \eqref{eq: holder in lorentz space} can be easily proved using the definitions given in Remark \ref{rem: norms via decr rearr}. For $q=\infty$ it follows directly from the definition of 
	$\|\cdot\|_{p, \infty}$, while for $q\in (0, \infty)$ it is obtained using inequality $\int f^\alpha g^\beta \frac{dt}{t} \leq (\int f\frac{dt}{t})^\alpha (\int g \frac{dt}{t})^\beta$ for $\alpha$ and $\beta$ which sum to 1. In particular taking $\alpha= \frac{q_2}{q_1+q_2}$ and $\beta= \frac{q_1}{q_1+q_2}$
	\begin{align*}
		\|fg\|_{p,q}^q =& \int (t^{\frac{1}{p}} f^* g^*)^q \frac{dt}{t} = \int [(t^{\frac{1}{p_1}} f^*)^{q_1} ]^{\frac{q_2}{q_1+q_2}} [(t^{\frac{1}{p_2}} f^*)^{q_2} ]^{\frac{q_1}{q_1+q_2}} \frac{dt}{t}\\
		\leq& (\int (t^{\frac{1}{p_1}} f^*)^{q_1}\frac{dt}{t})^{\frac{q_2}{q_1+q_2}} (\int (t^{\frac{1}{p_2}} f^*)^{q_2}\frac{dt}{t})^{\frac{q_1}{q_1+q_2}}
	\end{align*}
	from which\eqref{eq: holder in lorentz space} follows taking the power $\frac{1}{q}= \frac{q_1+q_2}{q_1q_2}$

	The following property is a sort of interpolation over the first index of the space.

	\begin{prop}
		\label{prop: f in intersection of L^p infty}
		Let $f \in L^{p_0, \infty} \cap L^{p_1, \infty} $ , then $f \in L^{p, q}$ for any $p\in (p_0, p_1)$ and any $q\in (0, \infty]$. 
	\end{prop}
	
	\begin{proof}
		We first consider $q= \infty$ then 
		\begin{equation}
			\label{comp: norm split in d_f<1 d_f>1}
			\| f\|_{p, \infty} = \max \{ \sup_{d_f(t)\geq 1} t d_f(t)^{1/p},\ \sup_{0\leq d_f(t)\leq 1} t d_f(t)^{1/p}\}.
		\end{equation}
		For the case $d_f(t)\leq 1$ since $p<p_1$ we have 
		\begin{equation*}
			d_f(t) ^{1/p- 1/p_1} \leq 1,
		\end{equation*}
		conversely when $d_f(t)\geq 1$ 
		\begin{equation*}
			d_f(t) ^{1/p- 1/p_0} \leq 1.
		\end{equation*}
		We can therefore bound both suprema in \eqref{comp: norm split in d_f<1 d_f>1} as 
		\begin{equation*}
			\sup_{0\leq d_f(t)\leq 1} t d_f(t)^{1/p-1/p_1 + 1/p_1} \leq \sup_{0\leq d_f(t)\leq 1} t d_f(t)^{1/p_1}=\|f\|_{p_1, \infty} <\infty
		\end{equation*}
		and 
		\begin{equation*}
			\sup_{d_f(t)\geq 1} t d_f(t)^{1/p-1/p_0 + 1/p_0} \leq \sup_{t\geq 0} t d_f(t)^{1/p_0} = \|f\|_{p_0, \infty}<\infty.
		\end{equation*} 
		
		Now let $q<\infty$. We have the quantity $td_f(t)^{1/p_1}$ which is bounded at infinity, then there exists a $\overline{t}>0$ such that 
		\begin{equation*}
			d_f(t) \lesssim \frac{1}{t^{p_1}}\ \textnormal{ for any }\ t\geq \overline{t}
		\end{equation*}
		so for large enough $t$ we have 
		\begin{equation*}
			t d_f(t)^{1/p_0} \lesssim  t^{1-p_1/p_0} \ \textnormal{ with }\ 1-\frac{p_1}{p_0} <0.
		\end{equation*}
		On the other hand $td_f(t)^{1/p_0}$ is bounded around 0 so there exists $\tilde{t}>0$ such that 
		\begin{equation*}
			d_f(t) \lesssim \frac{1}{t^{p_0}}\ \textnormal{ for any }\ 0< t\leq \tilde{t}
		\end{equation*}
		from which for small $t$ it holds
		\begin{equation*}
			t d_f(t)^{1/p_1} \lesssim  t^{1-p_0/p_1} \ \textnormal{ with }\ 1-\frac{p_0}{p_1} >0.
		\end{equation*}
		Now we can proceed to estimate the $L^{p,q}$ norm, for $\lambda \in (0,1)$ we write $\frac{1}{p}= (1-\lambda) \frac{1}{p_1} + \lambda \frac{1}{p_0}$
		\begin{align*}
			\|f\|_{p,q} =& \int_0^t (t^{1-\lambda + \lambda} d_f(t)^{(1-\lambda) \frac{1}{p_1} + \lambda \frac{1}{p_0}} )^q \frac{dt}{t} 
			= \int_0^{\tilde{t}} \ldots dt  + \int_{\tilde{t}}^{\overline{t}} \ldots dt   + \int_{\overline{t}}^\infty\ldots dt \\
			\leq& \sup_{t} (t d_f(t)^{1/p_0})^{\lambda q} \int_0^{\tilde{t}} (t d_f(t)^{1/p_1})^{(1-\lambda) q} \frac{dt}{t}\\
			& + \int_{\tilde{t}}^{\overline{t}} \ldots dt \\
			& + \sup_{t} (t d_f(t)^{1/p_1})^{(1-\lambda) q}\int_{\overline{t}}^\infty (t d_f(t)^{1/p_0})^{\lambda q} \frac{dt}{t}\\
			\lesssim & \|f\|_{p_0, \infty}^{\lambda q} \int_0^{\tilde{t}} t^{(1-p_0/p_1)(1-\lambda)q-1} dt \\
			& + \int_{\tilde{t}}^{\overline{t}} \ldots dt \\
			& + \|f\|_{p_1, \infty}^{(1-\lambda)q} \int_{\overline{t}}^\infty t^{(1-p_1/p_0)\lambda q-1}dt\\
			\simeq &\|f\|_{p_0, \infty}^{\lambda q} \  t^{(1-p_0/p_1)(1-\lambda)q}|_0^{\tilde{t}} +  \int_{\tilde{t}}^{\overline{t}} \ldots dt\\
			& + \|f\|_{p_1, \infty}^{(1-\lambda)q}\  t^{(1-p_1/p_0)\lambda q}|_{\overline{t}}^\infty
		\end{align*}
		where all the terms are finite since $(1-p_0/p_1)(1-\lambda)q>0$ and $(1-p_1/p_0)\lambda q<0$. 
	\end{proof}

	\section{Proof of Lemma \ref{lemma: algebraic ineq}.}
	\label{app: proof of ineq}
	
	In this appendix we give a proof of Lemma \ref{lemma: algebraic ineq}. 
	
	\begin{lemma}
		Let $N \in \N$ and $k =0, \ldots , N$ $l= 0, \ldots, k$. There exist $ c_{kl} \in \R$ such that 
		\begin{align*}
			\bigg| \frac{1}{|x-y|^{n-2}} - \frac{1}{|x|^{n-2}} \sum_{k=0}^N \sum_{\substack{l=0\\ k + l \leq N}}^k c_{kl}& \frac{|y|^{k+l}}{|x|^{k+l}} \cdot\frac{(x\cdot y )^{k-l}}{(|x||y|)^{k-l}}\bigg|\\
			&\lesssim  \frac{1}{|x-y|^{n-2}}  \left( \frac{|y|^{N + 1}}{|x|^{N + 1}} + \frac{|y|^{N + n-2}}{|x|^{N + n-2}}\right)
		\end{align*}
		for all $x, y \in \R^n \setminus \{0\}$ and $x \neq y $. 	
	\end{lemma}
	
	\begin{proof}
		We divide the proof according to the size of $|y|$ with respect to $|x|$. 
		
		We first consider the case $|y|\gg |x|$ and use the trivial bound $\left|\frac{(x\cdot y )}{|x||y|}\right| \leq 1$, then 
		\begin{align*}
			&\left| \frac{1}{|x-y|^{n-2}} - \frac{1}{|x|^{n-2}}  \sum_{k=0}^N \sum_{\substack{l=0\\ k + l \leq N}}^k c_{kl} \frac{|y|^{k+l}}{|x|^{k+l}} \cdot\frac{(x\cdot y )^{k-l}}{(|x||y|)^{k-l}}\right| \\
			&	\lesssim  \frac{1}{|x-y|^{n-2}} +\frac{1}{|x|^{n-2}} \sum_{k=0}^N \sum_{\substack{l=0\\ k + l \leq N}}^k\frac{|y|^{k+l}}{|x|^{k+l}} \lesssim  \frac{1}{|x-y|^{n-2}} +\frac{|y|^N}{|x|^{n-2}|x|^N} \\
			&=  \frac{1}{|x-y|^{n-2}} + \frac{|y|^{N + n-2}}{|x|^{N + n-2}}\cdot \frac{1}{|y|^{n-2}} \lesssim \frac{1}{|x-y|^{n-2}} \left( 1+ \frac{|y|^{N + n-2}}{|x|^{N + n-2}}\right)\\
			& \lesssim \frac{1}{|x-y|^{n-2}} \cdot\frac{|y|^{N + n-2}}{|x|^{N + n-2}}
		\end{align*}
		where we used the bounds $|x-y|\lesssim |y|$ and $1 \lesssim \frac{|y|}{|x|}$. 
		
		Now we consider the region $|y|\simeq |x|$. Again, $\left|\frac{(x\cdot y )}{|x||y|}\right| \leq 1$, moreover $\frac{|y|}{|x|} \simeq 1 $ and $|x-y|\lesssim |x|$. Therefore, we bound trivially by 
		\begin{align*}
			&\left| \frac{1}{|x-y|^{n-2}} - \frac{1}{|x|^{n-2}}  \sum_{k=0}^N \sum_{\substack{l=0\\ k + l \leq N}}^k c_{kl} \frac{|y|^{k+l}}{|x|^{k+l}} \cdot\frac{(x\cdot y )^{k-l}}{(|x||y|)^{k-l}}\right| \\
			&	\lesssim \frac{1}{|x-y|^{n-2}} + \frac{1}{|x|^{n-2}} \lesssim  \frac{1}{|x-y|^{n-2}} \simeq  \frac{1}{|x-y|^{n-2}} \cdot\frac{|y|^{N +1}}{|x|^{N + 1}}. 
		\end{align*}
		To conclude, we look at the case $|y|\ll |x|$. We have 
		\begin{align*}
			|x-y|^{n-2} = &(|x|^2 + |y|^2 - 2 x\cdot y )^{\frac{n-2}{2}} = |x|^{n-2} (1 + \frac{|y|^2}{|x|^2} - 2 \frac{x \cdot y }{|x|^2})^{\frac{n-2}{2}}  \nonumber\\
			=&   |x|^{n-2} (1 + \frac{|y|^2}{|x|^2} - 2\frac{|y|}{|x|} \frac{(x\cdot y )}{|x||y|})^{\frac{n-2}{2}},
		\end{align*}
		which yields 
		\begin{align}
			\frac{1}{ |x-y|^{n-2} } = \frac{1}{ |x|^{n-2} } \cdot \frac{1}{ (1 + \frac{|y|^2}{|x|^2} - 2\frac{|y|}{|x|} \frac{(x\cdot y )}{|x||y|})^{\frac{n-2}{2}}} \label{comp: |x-y|}
		\end{align}
		with $\frac{|y|}{|x|}\ll 1$. We use the Taylor expansion around zero for the function $ \frac{1}{(1+ s)^{\frac{n-2}{2}} }$ given by
		\begin{align*}
			\frac{1}{(1+ s)^{\frac{n-2}{2}} } = & \sum_{k=0}^N d_k s^k + O(|s|^{N+ 1}) \quad d_k := \frac{ 1}{k!} \frac{d}{d^k} \left(\frac{1}{(1+ s)^{\frac{n-2}{2}} } \right)|_{s=0}.
		\end{align*}
		We apply this expansion with $s =  \frac{|y|^2}{|x|^2} - 2\frac{|y|}{|x|} \frac{(x\cdot y )}{|x||y|}$ which satisfies the bound $|s|\ll 1$, given that we are in the region $|y|\ll |x|$. From \eqref{comp: |x-y|} we have  
		\begin{align*}
			\frac{1}{ |x-y|^{n-2} } =  & \frac{1}{ |x|^{n-2} } \cdot  \left(\sum_{k=0}^N d_k s^k + O(|s|^{N+ 1})\right)
		\end{align*}
		and given the choice of $s$ 
		\begin{align}
			\left| \frac{1}{ |x-y|^{n-2} }-\frac{1}{ |x|^{n-2} } \cdot  \sum_{k=0}^N d_k \left( \frac{|y|^2}{|x|^2} - 2\frac{|y|}{|x|} \frac{(x\cdot y )}{|x||y|}\right)^k \right|\lesssim &\frac{|s|^{N+ 1}}{|x|^{n-2}}\nonumber\\
			\lesssim & \frac{1}{ |x|^{n-2} } \left( \frac{|y|^2}{|x|^2} + \frac{|y|}{|x|}\right)^{N+ 1} \nonumber\\
			\lesssim & \frac{1}{ |x|^{n-2} } \cdot \frac{|y|^{N+ 1}}{|x|^{N+ 1}} \label{comp: bound taylos exp}
		\end{align}
		where in the last inequality we used the fact that $\frac{|y|^2}{|x|^2}\ll \frac{|y|}{|x|}$. 
		
		Finally, we remark that we can reorganize the sum in the previous inequality as 
		\begin{align}
			\sum_{k=0}^N d_k \left( \frac{|y|^2}{|x|^2} - 2\frac{|y|}{|x|} \frac{(x\cdot y )}{|x||y|}\right)^k = & \sum_{k=0}^N d_k \frac{|y|^k}{|x|^k} \sum_{l=0}^k c_{kl} \frac{|y|^l}{|x|^l} \cdot\frac{(x\cdot y )^{k-l}}{(|x||y|)^{k-l}} \nonumber\\
			=& \sum_{k=0}^N \sum_{\substack{l=0\\ k + l \leq N}}^k c_{kl} \frac{|y|^{k+l}}{|x|^{k+l}} \cdot\frac{(x\cdot y )^{k-l}}{(|x||y|)^{k-l}} \nonumber\\
			& + O\left(\frac{|y|^{N+1}}{|x|^{N+1}}\right)  \label{comp: expression taylor coeff}
		\end{align}
		We obtain the term in \eqref{comp: expression taylor coeff} thanks to the following remark:
		\begin{align*}
			\{(k, l ) \in \N^2: (k, l) \in [0,N]\times &[0,k]\ \textnormal{and }  k + l \leq N \}\\
			=&  \{(k, l ) \in\N^2: (k, l) \in [0, N/2]\times [0,k]\}\\ & \cup \{(k, l ) \in \N^2: (k, l) \in [N/2,N]\times [0, N-k]\} 
		\end{align*}
		and for $k = N/2, \ldots ,N$ and $l = N-k, \ldots, k$ we have $l + k \geq N+1$, so
		\begin{align*}
			\left|\frac{|y|^{k+l}}{|x|^{k+l}} \cdot\frac{(x\cdot y )^{k-l}}{(|x||y|)^{k-l}}\right| \lesssim \frac{|y|^{N+1}}{|x|^{N+1}}
		\end{align*}  
		since $|y|\ll |x|$. Thanks to \eqref{comp: bound taylos exp} and \eqref{comp: expression taylor coeff} we obtain the statement, since $|x-y|\lesssim |x|$. 
		
	\end{proof}

	\bibliography{biblio.bib}
	
	\textsc{Institut de Mathématiques de Toulouse}, 118 route de Narbonne, Toulouse, F-31062 Cedex 9, France. \textit{email}: \texttt{viviana.grasselli@math.univ-toulouse.fr}

\end{document}